\documentclass[a4paper,12pt,eqno]{article}

\usepackage{amsthm,amsmath,amssymb}

\usepackage{graphicx}

\usepackage{color}

\theoremstyle{plain}
\newtheorem{theorem}{Theorem}[section]
\newtheorem{lemma}[theorem]{Lemma}
\newtheorem{corollary}[theorem]{Corollary}
\newtheorem{proposition}[theorem]{Proposition}

\theoremstyle{definition}

\theoremstyle{remark}
\newtheorem{remark}[theorem]{Remark}

\makeatletter
 \@addtoreset{equation}{section}
\makeatother

\newcommand{\RM}{\mathbb{R}}

\newcommand{\NM}{\mathbb{N}}
\newcommand{\CM}{\mathbb{C}}

\newcommand{\HM}{\mathbb{H}}
\newcommand{\Mat}{\operatorname{Mat}}
\newcommand{\Sdet}{\operatorname{Sdet}}
\newcommand{\RP}{\operatorname{Re}}
\newcommand{\laangle}{\langle\!\langle}
\newcommand{\raangle}{\rangle\!\rangle}

\title{\bf The quaternionic second weighted zeta function of a graph and 
the Study determinant}

\author{
{\small Norio Konno}\\
{\scriptsize Department of Applied Mathematics, 
Faculty of Engineering, 
Yokohama National University}\\
{\scriptsize Hodogaya, Yokohama 240-8501, Japan}\\
{\scriptsize e-mail: konno@ynu.ac.jp}\\
{\small Hideo Mitsuhashi}\\
{\scriptsize Faculty of Education, 
Utsunomiya University}\\
{\scriptsize Utsunomiya, Tochigi 321-8505, Japan}\\
{\scriptsize e-mail: mitsu@cc.utsunomiya-u.ac.jp}\\
{\small Iwao Sato}\\
{\scriptsize Oyama National College of Technology}\\
{\scriptsize Oyama, Tochigi 323-0806, Japan}\\
{\scriptsize e-mail: isato@oyama-ct.ac.jp}\\}

\date{\small Mathematics Subject Classifications: 05C50, 15A15, 11R52}

\begin{document}

\maketitle

\begin{abstract}
We establish a generalization of the second weighted zeta function of a graph 
to the case of quaternions. 
For an arc-weighted graph whose weights are quaternions, 
we define the second weighted zeta function by using the Study determinant 
that is a quaternionic determinant for quaternionic matrices defined by Study.   
This definition is regarded as a quaternionic analogue of the determinant expression of 
Hashimoto type for the Ihara zeta function of a graph. 
We derive the Study determinant expression of Bass type and the Euler product for 
the quaternionic second weighted zeta function. 

\end{abstract}

\setcounter{equation}{0}
\section{Introduction}\label{SecIntro}
The Ihara zeta function of a graph has achieved success in spectral theory of graphs. 
Zeta functions of graphs started from Ihara zeta functions of regular 
graphs by Ihara \cite{Ihara1966}. 
Ihara defined the $p$-adic Selberg zeta function ${\bf Z}_{\Gamma}(u)$ of 
a torsion-free discrete cocompact subgroup $\Gamma$ of $\operatorname{PGL}_2$ over a locally 
compact field under a discrete valuation, 
and showed that its reciprocal is a explicit polynomial. 
Ihara's motivation to define ${\bf Z}_{\Gamma}(u)$ was to 
count the number of primitive conjugacy classes of torsion-free discrete 
cocompact subgroups. 
Regarding Ihara's work, Serre \cite{Serre} pointed out that the Ihara zeta function 
is the zeta function of the quotient $\Gamma{\backslash}T$ (a finite regular graph) 
of the one-dimensional Bruhat-Tits building $T$ (an infinite regular tree) 
associated with $\operatorname{SL}_2(\mathbb{Q}_p)$. 
This observation led to rapid developments of zeta functions of graphs. 
Sunada \cite{Sunada1986}, \cite{Sunada1988} developed zeta functions of regular graphs 
equipped with unitary representations of fundamental groups of the graphs. 
Hashimoto \cite{Hashimoto1989} explored multivariable zeta functions of bipartite graphs and 
also gave a determinant expression for the Ihara zeta function of a general graph 
by using its edge matrix. 
Bass \cite{Bass1992} generalized Ihara's result on zeta functions of 
regular graphs to irregular graphs, and showed that their reciprocals are 
again polynomials. Subsequently, various proofs of Bass' Theorem were given by 
Stark and Terras \cite{ST1996}, Foata and Zeilberger \cite{FZ1999}, Kotani and Sunada \cite{KS2000}. 
Zeta functions of edge-weighted graphs were proposed by Hashimoto \cite{Hashimoto1990} and 
those of arc-weighted graphs by Stark and Terras \cite{ST1996} which are called edge zeta functions.  
Stark and Terras \cite{ST1996} gave their determinant expressions by using their edge matrices. 
Mizuno and Sato \cite{MS2004} focused on a special version of edge zeta functions, 
and defined the weighted zeta function by incorporating  
a variable $t$ that measures the length of cycles into the edge zeta function. 
Subsequently, Sato \cite{Sato} also defined a new class of zeta functions of graphs 
by modifying the determinant expression of the weighted zeta function defined in \cite{MS2004}. 
This new zeta function, which was named the second weighted zeta function by Sato, 
played essential roles in the concise proof of the spectral mapping theorem 
for the Grover walk on a graph in \cite{KS2012} and of the Smilansky's formula 
\cite{Smilansky2007} for the characteristic polynomial of the bond scattering matrix of a graph 
in \cite{MS2008}. 
Thereby we expect that the second weighted zeta function brings about 
rich outcomes in quantum dynamics on graphs such as quantum walks on graphs or quantum graphs. 
In this paper, we aim to establish a quaternionic analogue of 
the second weighted zeta function of a graph and to derive its basic properties. 
Our results will play crucial roles in our future work on quaternionic quantum walks which 
was established by Konno \cite{Konno2015} recently.

The quaternion was discovered by Hamilton in 1843. 
It can be considered as an extension of the complex number. However, 
quaternions do not commute mutually in general and 
the definition of determinant is invalid for quaternionic matrices. 
For many years, a number of researchers, for example Cayley, 
Study, Moore, Dieudonn{\'e}, Dyson, Mehta, Xie, Chen, 
have given different definitions of determinants of quaternionic matrices. 
Detailed accounts on the determinants of  
quaternionic matrices can be found in, for example, \cite{Aslaksen1996,Zhang1997}. 
In this paper, we extend the second weighted zeta function of a graph 
to the case of quaternions by using the quaternionic determinant 
defined by Study \cite{Study1920}. 
An advantage of this approach is that one can reduce a calculation 
of the Study determinant to that of the ordinary determinant, so is 
easier to handle and apply than other general and abstract approaches. 
The Study determinant enables us to derive the explicit determinant expression 
and the Euler product for the quaternionic second weighted zeta function 
analogous to the ordinary determinant.

The rest of the paper is organized as follows. 
In Section 2, we provide a summary of the Ihara zeta function and its variants. 
We give various zeta functions of graphs, including the second weighted zeta function, 
and present their determinant expressions. 
In the end of this section, we explain briefly that 
the second weighted zeta function can be viewed as a natural generalization 
of the Ihara zeta function of a tree lattice and is related 
to quantum systems on graphs. 
In Section 3, we explain the Study determinant of a quaternionic matrix and give some 
properties of it which are needed in later sections. 
In Section 4, we define the quaternionic second weighted zeta function of a graph by using 
the Study determinant which is considered as a quaternionic analogue of the determinant expression of 
Hashimoto type, and determine its determinant expression of Bass type 
(Theorem \ref{DetExpressionOfBassType}). 
In Section 5, we derive the Euler product for the quaternionic second weighted zeta 
function of a graph. 

\section{The Ihara zeta function of a graph} 
In this section, we provide a summary of the Ihara zeta function of a graph and its development which 
led the Ihara zeta function to the second weighted zeta function. 
Let $G=(V(G)$, $E(G))$ be a finite connected graph with the set $V(G)$ of 
vertices and the set $E(G)$ of undirected edges $uv$ 
joining two vertices $u$ and $v$. 
We assume that $G$ has neither loops nor multiple edges throughout. 
For $uv \in E(G)$, an arc $(u,v)$ is the directed edge from $u$ to $v$. 
Let $D(G)=\{\,(u,v),\,(v,u)\,\mid\,uv{\;\in\;}E(G)\}$ and 
$|V(G)|=n,\;|E(G)|=m,\;|D(G)|=2m$. 
For $e=(u,v){\;\in\;}D(G)$, $o(e)=u$ denotes the {\it origin} and $t(e)=v$ the {\it terminal} 
of $e$ respectively. 
Furthermore, let $e^{-1}=(v,u)$ be the {\em inverse} of $e=(u,v)$. 
The {\em degree} $\deg v = \deg {}_G \  v$ of a vertex $v$ of $G$ is the number of edges 
incident to $v$. 
We denote by $D_G$ the symmetric digraph whose vertex set is $V(G)$ and directed edge set 
is $D(G)$. 
A {\em path $P$ of length $\ell$} in $G$ is a sequence 
$P=(e_1, \cdots ,e_{\ell})$ of $\ell$ arcs such that $e_r \in D(G)$ and 
$t(e_r)=o(e_{r+1})$ for $r{\;\in\;}\{1,\cdots,\ell-1\}$. 
We set $o(P)=o(e_1)$ and $t(P)=t(e_{\ell})$. $|P|$ denotes the length of $P$. 
We say that a path $P=(e_1, \cdots ,e_{\ell})$ has a {\em backtracking} 
if $ e_{r+1} =e_r^{-1} $ for some $r\ (1{\leq}r{\leq}\ell-1)$, and that 
$P=(e_1, \cdots ,e_{\ell})$ has a {\em tail} if $ e_{\ell} =e_1^{-1} $. 
A path $P$ is said to be a {\em cycle} if $t(P)=o(P)$. 
The {\em inverse} of a path 
$P=(e_1,\cdots,e_{\ell})$ is the path 
$(e_{\ell}^{-1},\cdots,e_1^{-1})$ and is denoted by $P^{-1}$.

Two cycles $C_1 =(e_1, \cdots ,e_{\ell})$ and 
$C_2 =(f_1, \cdots ,f_{\ell})$ are said to be {\em equivalent} if there exists 
$s$ such that $f_r =e_{r+s}$ for all $r$ where indices are treated modulo $\ell$. 
Let $[C]$ be the equivalence class which contains the cycle $C$. 
Let $B^r$ be the cycle obtained by going $r$ times around a cycle $B$. 
Such a cycle is called a {\em power} of $B$. 
A cycle $C$ is said to be {\em reduced} if 
both $C$ and $C^2$ have no backtracking. 
Furthermore, a cycle $C$ is said to be {\em prime} if it is not a power of 
a strictly smaller cycle.

The {\em Ihara zeta function} of a graph $G$ is 
a function of $t \in {\bf C}$ with $|t|$ sufficiently small, 
defined by 
\[
{\bf Z} (G, t)= {\bf Z}_G (t)= \prod_{[C]} (1- t^{ \mid C \mid } )^{-1} ,
\]
where $[C]$ runs over all equivalence classes of prime, reduced cycles of $G$. 
Ihara's original definition was group theoretic and defined for 
a torsion-free discrete cocompact subgroup $\Gamma$ of $\operatorname{PGL}_2$ 
over a locally compact field under a discrete valuation. 
In the case of regular graphs, 
equivalence classes of prime, reduced cycles of $G$ correspond to 
primitive conjugacy classes of $\Gamma$ and $|C|$ the degree of the 
corresponding primitive conjugacy class. 
For the details, see \cite{Ihara1966}. 
Determinant expressions of Ihara zeta functions for finite graphs 
are obtained in the following way.

Let ${\bf B}=( {\bf B}_{ef} )_{e,f \in D(G)} $ and 
${\bf J}_0=( {\bf J}_{ef} )_{e,f \in D(G)}$ be $2m \times 2m$ matrices 
defined as follows: 
\[
{\bf B}_{ef} =\left\{
\begin{array}{ll}
1 & \mbox{if $t(e)=o(f)$, } \\
0 & \mbox{otherwise,}
\end{array}
\right.
{\bf J}_{ef} =\left\{
\begin{array}{ll}
1 & \mbox{if $f= e^{-1} $, } \\
0 & \mbox{otherwise.}
\end{array}
\right.
\]
Then the matrix ${\bf B} - {\bf J}_0 $ is called the {\em edge matrix} of $G$.

\begin{theorem}[Hashimoto \cite{Hashimoto1989}; Bass \cite{Bass1992}]
Let $G$ be a connected graph. 
Then the reciprocal of the Ihara zeta function of $G$ is given by 
\begin{equation}\label{DetExpressionsOfIharaZeta}
{\bf Z} (G, t)^{-1} =\det ( {\bf I}_{2m} -t ( {\bf B} - {\bf J}_0 ))
=(1- t^2 )^{r-1} \det ( {\bf I}_n -t {\bf A}+ 
t^2 ({\bf D} -{\bf I}_n )), 
\end{equation}
where $r$ and ${\bf A}$ are the Betti number and the adjacency matrix 
of $G$ respectively, and ${\bf D} =({\bf D}_{uv})_{u,v{\in}V(G)}$ is the diagonal matrix 
with ${\bf D}_{uu} = \deg u $ for all $u{\;\in\;}V(G)$. 
\end{theorem}

We call the middle formula the determinant expression 
of {\it Hashimoto type} and the right hand side 
the determinant expression of {\it Bass type} in (\ref{DetExpressionsOfIharaZeta}). 

Now we shall give the definition of the second weighted zeta function. 
Consider an $n \times n$ complex matrix 
${\bf W} =({\bf W}_{uv})_{u,v{\in}V(G)}$ with $(u,v)$-entry 
equals $0$ if $(u,v){\;\notin\;}D(G)$. 
We call ${\bf W}$ a {\em weighted matrix} of $G$.
Furthermore, let $w(u,v)= {\bf W}_{uv}$ for $u,v \in V(G)$ and 
$w(e)= w(u,v)$ if $e=(u,v) \in D(G)$. 
For a path $P=( e_1 , \cdots , e_{\ell} )$ of $G$, the {\em norm} $w(P)$ of $P$ 
is defined by $w(P)=w(e_1)w(e_2){\cdots}w(e_{\ell})$. 
For a weighted matrix ${\bf W}$ of $G$, let 
${\bf B}_w=( {\bf B}^{(w)}_{ef} )_{e,f \in D(G)}$ be 
the $2m \times 2m$ complex matrix as follows: 
\begin{equation*}
{\bf B}^{(w)}_{ef} =\left\{
\begin{array}{ll}
w(f) & \mbox{if $t(e)=o(f)$, } \\
0 & \mbox{otherwise.}
\end{array}
\right.
\end{equation*}
Then the {\em second weighted zeta function} of $G$ is defined by 
\[
{\bf Z}_1 (G,w,t)= \det ( {\bf I}_{2m} -t ( {\bf B}_w - {\bf J}_0 ) )^{-1} . 
\]
We call ${\bf B}_w - {\bf J}_0$ the {\it ${\bf B}$-weighted edge matrix} of $G$. 
If $w(e)=1$ for any $e \in D(G)$, then the second weighted zeta function of $G$ 
coincides with the Ihara zeta function of $G$.

\begin{theorem}[Sato \cite{Sato}]\label{SatoThm}
Let $G$ be a connected graph, and 
let ${\bf W}$ be a weighted matrix of $G$. 
Then the reciprocal of the second weighted zeta function of $G$ is given by 
\[
{\bf Z}_1 (G,w,t )^{-1} =(1- t^2 )^{m-n} 
\det ({\bf I}_n -t {\bf W}+ t^2 ( {\bf D}_w - {\bf I}_n )) , 
\]
where $n=|V(G)|$, $m=|E(G)|$ and 
${\bf D}_w =({\bf D}^{(w)}_{uv})_{u,v{\in}V(G)}$ is the diagonal matrix 
with ${\bf D}^{(w)}_{uu} = \displaystyle \sum_{e:o(e)=u}w(e)$ for all 
$u{\;\in\;}V(G)$. 
\end{theorem}

In \cite{MS2008}, Mizuno and Sato obtained the Euler product for ${\bf Z}_1 (G,w,t )$. 
Let $\tilde{w}(e, f)$ be the $(e,f)$-entry of the matrix ${\bf B}_w-{\bf J}_0$.
$\tilde{w}(e, f)$ is given by the following formula: 
\begin{equation}\label{EntriesOf2ndEdgeMatrix}
\tilde{w}(e, f)=\begin{cases}
w(f) & \text{if $t(e)=o(f)$ and $f{\;\neq\;}e^{-1}$},\\
w(f)-1 & \text{if $f=e^{-1}$},\\
0 & \text{otherwise}.
\end{cases}
\end{equation}
We set $\tilde{w}(P)=\tilde{w}(e_1, e_2)\tilde{w}(e_2, e_3){\cdots}\tilde{w}(e_{\ell-1}, e_{\ell})$ 
for a path $P=( e_1 , \cdots , e_{\ell} )$. 

\begin{theorem}[Mizuno and Sato \cite{MS2008}]\label{MizunoSatoThm}
Let G be a connected graph. Then
\[
{\bf Z}_1 (G,w,t )=\prod_{[C]}(1-\tilde{w}(C)t^{|C|})^{-1}, 
\]
where $[C]$ runs over all equivalence classes of prime cycles of $G$. 
\end{theorem}

We shall give some notable comments on the second weighted zeta function of a graph. 
Let $T$ be a locally finite (possibly infinite) tree 
and $\Gamma$ a group with an action on the symmetric digraph 
$X=D_T=(V(X),\overrightarrow{E}(X))$ where $V(X)=V(T)$ and 
$\overrightarrow{E}(X)=D(T)$. 
According to \cite{Bass1992}, we assume that $\Gamma$ satisfies the following conditions:  
\begin{description}
\item[(\hspace{0.9mm}I\hspace{0.9mm})]
$\Gamma$ acts without inversions, that is, 
${\Gamma}e{\;\neq\;}{\Gamma}e^{-1}$ for all $e{\;\in\;}\overrightarrow{E}(X)$. 
\item[(D)]
$\Gamma$ is discrete, that is, the stabilizer $\Gamma_u$ of $u$ 
is finite for every $u{\;\in\;}V(X)$. 
\item[(\hspace{0.4mm}F\hspace{0.4mm})]
$\Gamma$ is uniform (= cocompact), that is, $\Gamma{\backslash}X$ is finite. 
\end{description} 
$\Gamma$ is called a {\it uniform tree lattice}. 
Let $Y=(V(Y),\overrightarrow{E}(Y))=\Gamma{\backslash}X$ be the quotient digraph and 
$p:X{\longrightarrow}Y$ the projection. 
For $x{\;\in\;}X$ and $y=p(x){\;\in\;}Y$, the {\it index} $i(e)$ of 
$e{\;\in\;}\overrightarrow{E}(Y)$ which satisfies $o(e)=y$ 
is defined by  
\[
i(e)=|\{\,e'{\;\in\;}\overrightarrow{E}(X)\;|\;o(e')=x,\,p(e')=e\,\}|.
\]
The pair $(Y,i)$ is called an {\it edge-indexed graph}. 
Now we assume $Y$ has no loop and 
set $D_G=Y$ and $w(e)=i(e)$ for all $e{\;\in\;}\overrightarrow{E}(Y)$. 
Then ${\bf Z}_1 (G,w,t)$ coincides with the zeta function of the edge-indexed 
graph $(Y,i)$ which was defined by Bass \cite{Bass1992}. 
Therefore we can consider ${\bf Z}_1 (G,w,t)$ as a natural generalization 
of the zeta function of the edge-indexed graph. 

The second weighted zeta function has remarkable connections with quantum dynamics 
on graphs. 
Replacing $t$ with $1/{\lambda}$ in Theorem \ref{SatoThm} we have
\begin{equation}\label{EqnHasimotoBass}
\det ( \lambda{\bf I}_{2m} -( {\bf B}_w - {\bf J}_0 ) )=
(\lambda^2- 1 )^{m-n} 
\det (\lambda^2{\bf I}_n -\lambda{\bf W}+ ( {\bf D}_w - {\bf I}_n )).
\end{equation}
Using this equation, several quantum systems on graphs have been 
investigated. Consequently, new proofs of spectral properties of them were given in 
\cite{MS2008,KS2012} and spectra were newly determined by easy derivable parameters 
from ${\bf W}$ in \cite{HKSS2013,KMS2016}. 
Interestingly, time evolution operators of important quantum systems on graphs, 
for example, the Grover walk, the Szegedy walk, the bond scattering matrix of Smilansky, 
can be expressed by ${\bf B}$-weighted edge matrices. It enables us to apply (\ref{EqnHasimotoBass}) 
to derive spectra of time evolution operators. 
In this way, the generalization of Ihara zeta functions of tree lattices 
has significant applications in quantum mechanics. 
We expect that the second weighted zeta function brings about 
further outcomes in quantum dynamics on graphs such as quantum walks or quantum graphs.

\section{The Study determinant of a quaternionic matrix}
We shall establish the quaternionic second weighted zeta function of a graph 
from now on.  
In this section, we explain quaternions and quaternionic matrices that are 
needed in later sections. 
Let $\HM$ be the set of quaternions. $\HM$ is a noncommutative associative 
algebra over $\RM$, whose underlying real vector space has dimension $4$ 
with a basis $1,i,j,k$ which satisfy the following relations: 
\[
i^2=j^2=k^2=-1,\quad ij=-ji=k,\quad jk=-kj=i,\quad ki=-ik=j.
\]
For $x=x_0+x_1i+x_2j+x_3k{\;\in\;}\HM$, 
$x^*$ denotes the conjugate of $x$ in $\HM$ 
which is defined by $x^*=x_0-x_1i-x_2j-x_3k$, 
and $\RP x=x_0$ the real part of $x$. 
One can easily check $xx^*=x^*x$, $(x^*)^n=(x^n)^*$, and 
$x^{-1}=x^*/|x|^2$ for $x{\;\neq\;}0$. 
Hence, $\HM$ constitutes a skew field. 
We call $|x|=\sqrt{xx^*}=\sqrt{x^*x}=\sqrt{x_0^2+x_1^2+x_2^2+x_3^2}$ the norm of $x$. 
Indeed, $|{\,\cdot\,}|$ satisfies \\
\quad (1) $|x|{\;\geq\;}0$, and moreover $|x|=0 \Leftrightarrow x=0$,\\
\quad (2) $|xy|=|x||y|$,\\
\quad (3) $|x+y|{\;\leq\;}|x|+|y|$.\\
We notice that $\HM$ is a complete space with respect to the metric derived from 
the norm of quaternions. 

Any quaternion $x$ can be presented by two complex numbers 
$x=a+jb$ uniquely. 
Explicitly, if $x=x_0+x_1i+x_2j+x_3k$ then $a=x_0+x_1i$ and $b=x_2-x_3i$. 
Such a presentation is called the {\it symplectic decomposition}. 
Two complex numbers $a$ and $b$ are called the {\it simplex part} and the 
{\it perplex part} of $x$ respectively. 
We mean by a {\it quaternionic matrix} a matrix whose entries are quaternions. 
The symplectic decomposition is also valid for a quaternionic matrix. 
$\Mat(m{\times}n,\HM)$ denotes the set of $m{\times}n$ quaternionic matrices and 
$\Mat(m,\HM)$ the set of $m{\times}m$ quaternionic matrices. 
For ${\bf M}{\;\in\;}\Mat(m{\times}n,\HM)$, we can write ${\bf M}={\bf M}^S+j{\bf M}^P$ 
uniquely where ${\bf M}^S,{\bf M}^P{\;\in\;}\Mat(m,n,\CM)$. 
${\bf M}^S$ and ${\bf M}^P$ are called the {\it simplex part} and the 
{\it perplex part} of ${\bf M}$ respectively. 
We define $\psi$ to be the map from $\Mat(m{\times}n,\mathbb{H})$ to 
$\Mat(2m{\times}2n,\mathbb{C})$ as follows: 
\[
\psi : \Mat(m{\times}n,\mathbb{H}){\;\longrightarrow\;}\Mat(2m{\times}2n,\mathbb{C})
\quad{\bf M}{\;\mapsto\;}\begin{bmatrix}{\bf M}^S&-\overline{{\bf M}^P}\\{\bf M}^P&
\overline{{\bf M}^S}\end{bmatrix},
\]
where $\overline{\bf A}$ is the complex conjugate of a complex matrix ${\bf A}$. 
Then $\psi$ is an $\RM$-linear map. We also have 

\begin{lemma}\label{PsiLem}
Let ${\bf M}{\;\in\;}\Mat(m{\times}n,\HM)$ and ${\bf N}{\;\in\;}
\Mat(n{\times}m,\HM)$. Then 
\[ \psi({\bf M}{\bf N})=\psi({\bf M})\psi({\bf N}). \]
\end{lemma}

\begin{proof} 
Let ${\bf M}={\bf A}+j{\bf B}$ and ${\bf N}={\bf C}+j{\bf D}$ 
be symplectic decompositions of ${\bf M}$ and ${\bf N}$. Then 
\[ {\bf M}{\bf N}=({\bf A}+j{\bf B})({\bf C}+j{\bf D})
={\bf A}{\bf C}+{\bf A}j{\bf D}+j{\bf B}{\bf C}+j{\bf B}j{\bf D}. \]
Since  
${\bf X}j=j\overline{\bf X}$ for every complex matrix ${\bf X}$, 
we obtain 
\[ {\bf M}{\bf N}={\bf A}{\bf C}-\overline{{\bf B}}{\bf D}
+j(\overline{{\bf A}}{\bf D}+{\bf B}{\bf C}), \]
and therefore
\[ \psi({\bf M}{\bf N})
=\begin{bmatrix}{\bf A}{\bf C}-\overline{{\bf B}}{\bf D}
&-{\bf A}\overline{{\bf D}}-\overline{{\bf B}}\overline{{\bf C}}\\
\overline{{\bf A}}{\bf D}+{\bf B}{\bf C}
&\overline{{\bf A}}\overline{{\bf C}}-{\bf B}\overline{{\bf D}}
\end{bmatrix}.
\]
On the other hand, 
\[ \psi({\bf M})\psi({\bf N})=
\begin{bmatrix}{\bf A}&-\overline{{\bf B}}\\
{\bf B}&\overline{{\bf A}}
\end{bmatrix}
\begin{bmatrix}{\bf C}&-\overline{{\bf D}}\\
{\bf D}&\overline{{\bf C}}
\end{bmatrix}
=\begin{bmatrix}{\bf A}{\bf C}-\overline{{\bf B}}{\bf D}
&-{\bf A}\overline{{\bf D}}-\overline{{\bf B}}\overline{{\bf C}}\\
{\bf B}{\bf C}+\overline{{\bf A}}{\bf D}
&-{\bf B}\overline{{\bf D}}+\overline{{\bf A}}\overline{{\bf C}}
\end{bmatrix}.
\]
Thus $\psi({\bf M}{\bf N})=\psi({\bf M})\psi({\bf N})$ holds. 
\end{proof}

From Lemma \ref{PsiLem}, it follows immediately that

\begin{proposition}
If $m=n$, then $\psi$ is an injective $\mathbb{R}$-algebra homomorphism. 
\end{proposition}

In \cite{Study1920}, Study defined a determinant of an $n{\times}n$ 
quaternionic matrix which we denote by $\Sdet$, 
by setting $\Sdet({\bf M})=\det(\psi({\bf M}))$, 
where $\det$ is the ordinary determinant. 
We call $\Sdet$ the {\em Study determinant}. 
The Study determinant is the unique, up to a real power factor, functional $d_{\HM}$ 
which satisfies the following three axioms \cite{Aslaksen1996}: 
\begin{description}
\item (A1) $d_{\HM}({\bf A})=0{\;\Leftrightarrow\;}{\bf A}$ is singular.
\item (A2) $d_{\HM}({\bf AB})=d_{\HM}({\bf A})d_{\HM}({\bf B})$ for all ${\bf A,B}{\in}\Mat(n,\HM)$.
\item (A3) If ${\bf A}'$ is obtained from ${\bf A}$ by adding a left-multiple of a row to 
another row or a right-multiple of a column to another column, 
then $d_{\HM}({\bf A}')=d_{\HM}({\bf A})$.
\end{description}
Therefore it is reasonable to adopt the Study determinant to investigate 
the quaternionization of determinant expressions for zeta functions of graphs.  
Before stating properties of $\Sdet$, we mention a useful formula 
whose proof can be found in for example \cite{Zhang2011}: 

\begin{lemma}\label{DetPartitionedMat}
If ${\bf A}, {\bf B}, {\bf C}, {\bf D}$ are complex square matrices with same size 
and ${\bf A}{\bf C}={\bf C}{\bf A}$, then 
\[
\det  
\left[ 
\begin{array}{cc}
{\bf A} & {\bf B} \\
{\bf C} & {\bf D}   
\end{array} 
\right] 
= \det ({\bf A}{\bf D}-{\bf C}{\bf B}) . 
\]
\end{lemma}

$\Sdet$ has several basic properties as follows: 

\begin{proposition}\label{SdetProperties}{\ }
\renewcommand{\labelenumi}{\rm (\roman{enumi})}
\begin{enumerate}
\item
$\Sdet({\bf M}){\;\in\;}\RM_{\geq 0}=\{a{\in}\RM{\;\mid\;}a{\;\geq\;}0\}$ 
for ${\bf M}{\;\in\;}\Mat(n,\HM)$. 

\item
$\Sdet({\bf M})=0{\ \Leftrightarrow\ }{\bf M} \text{ has no inverse.}$

\item
$\Sdet({\bf MN})=\Sdet({\bf M})\Sdet({\bf N})$ for 
${\bf M},{\bf N}{\;\in\;}\Mat(n,\HM)$.

\item
If ${\bf N}$ is obtained from ${\bf M}$ by adding a left-multiple of a row 
to another row or a right-multiple of a column to another column, 
then $\Sdet({\bf N})=\Sdet({\bf M})$. 

\item
$\Sdet(\alpha{\bf M})=\Sdet({\bf M}\alpha)=|\alpha|^{2n}\Sdet({\bf M})$ 
for ${\bf M}{\;\in\;}\Mat(n,\HM),\,\alpha{\;\in\;}\HM$. 

\item
If ${\bf M}{\;\in\;}\Mat(n,\HM)$ is of the form:
\[
{\bf M}=
\begin{bmatrix}
\lambda_1 & * &{\cdots}& * \\
0 & \lambda_2 & & * \\
{\vdots}& & \ddots &{\vdots}\\
0 & 0 &{\cdots}&\lambda_n
\end{bmatrix} \quad\text{or}\quad 
\begin{bmatrix}
\lambda_1 & 0 &{\cdots}& 0 \\
* & \lambda_2 & & 0 \\
{\vdots}& & \ddots &{\vdots}\\
* & * &{\cdots}&\lambda_n
\end{bmatrix},
\]
then $\Sdet({\bf M})=\prod_{r=1}^n|\lambda_r|^2$.

\item  Let ${\bf A}$ be an $m \times n$ matrix and ${\bf B}$ an $n \times m$ matrix.
Then 
\[
{\rm Sdet} ({\bf I}_m -{\bf A}{\bf B})= {\rm Sdet} ({\bf I}_n -{\bf B}{\bf A}) . 
\]
\end{enumerate}
\end{proposition}

\begin{proof}
The proofs of (i), (ii), (iii), (iv) can be found in \cite{Aslaksen1996}. 
We prove (v), (vi) and (vii). 

\noindent
(v)\ \ Let ${\bf M} \in \Mat(n,\HM)$ and 
$\alpha = \alpha {}_s +j \alpha {}_p \in {\HM} ( \alpha {}_s , \alpha {}_p \in {\CM} )$. 
Then using Lemma \ref{PsiLem} and Lemma \ref{DetPartitionedMat}, we have 
\begin{equation*}
\begin{split}
{\Sdet} ( \alpha {\bf M}) &= \det ( \psi ( \alpha {\bf M}))
=\det ( \psi ( \alpha {\bf I}_n ) \psi ({\bf M})) \\
&=\det  
\begin{bmatrix}
\alpha {}_s {\bf I}_n & - \overline{ \alpha } {}_p {\bf I}_n \\ 
\alpha {}_p {\bf I}_n & \overline{ \alpha } {}_s {\bf I}_n  
\end{bmatrix}
\det (\psi ({\bf M})) \\ 
&=  
\det (( \alpha {}_s {\bf I}_n )( \overline{ \alpha } {}_s {\bf I}_n )
+( \alpha {}_p {\bf I}_n )( \overline{ \alpha } {}_p {\bf I}_n )) 
{\Sdet} ({\bf M}) \\ 
&= 
\det (( | \alpha {}_s |^2 +| \alpha {}_p |^2 ) {\bf I}_n ) {\rm Sdet} ({\bf M}) \\ 
&= | \alpha |^{2n} {\Sdet} ({\bf M}). 
\end{split}
\end{equation*}
In the same way, we can deduce $\Sdet({\bf M}\alpha)=|\alpha|^{2n}\Sdet({\bf M})$. 

\noindent
(vi)\ \ For a $2n{\times}2n$ matrix ${\bf N}$ and any two subsets 
$I=\{i_1,i_2,{\cdots},i_r\},\, J=\{j_1,j_2,{\cdots},j_s\}$ of $[2n]$, 
${\bf N}^{IJ}$ denotes the submatrix obtained from ${\bf N}$ by deleting 
$i_1,i_2,{\cdots},i_r$ th rows and $j_1,j_2,{\cdots},j_s$ th columns. 
Then by definitions of $\Sdet$ and $\psi$, we get the following: 
\begin{equation*}
\begin{split}
\Sdet({\bf M})&=
{\det}(\psi({\bf M}))
={\det} \begin{bmatrix}{\bf M}^S & -\overline{{\bf M}^P}\\
{\bf M}^P & \overline{{\bf M}^S}
\end{bmatrix}\\
&={\det} \begin{bmatrix}
\lambda_1^{S}& * & {\cdots} & * & -\overline{\lambda_1^{P}} & * & {\cdots} & * \\
0 &\lambda_2^{S}&  & * & 0 & -\overline{\lambda_2^{P}} & & * \\
{\vdots} &   & \ddots & {\vdots} & {\vdots} &   & \ddots & {\vdots} \\
0 & 0 & {\cdots} &\lambda_n^{S} & 0 & 0 & {\cdots} & -\overline{\lambda_n^{P}} \\
\lambda_1^{P} & * & {\cdots} & * & \overline{\lambda_1^{S}}& * & {\cdots} & * \\
0 & \lambda_2^{P} &  & * & 0 &\overline{\lambda_2^{S}}& & * \\
{\vdots} &   & \ddots & {\vdots} & {\vdots} &   & \ddots & {\vdots} \\
0 & 0 & {\cdots} & \lambda_n^{P} & 0 & 0 & {\cdots} &\overline{\lambda_n^{S}}
\end{bmatrix}\\
&=\lambda_1^{S}\det(\psi({\bf M})^{\{1\}\{1\}})
+(-1)^{n+2}\lambda_1^{P}\det(\psi({\bf M})^{\{n+1\}\{1\}})\\
&=\lambda_1^{S}\overline{\lambda_1^{S}}\det(\psi({\bf M})^{\{1,n+1\}\{1,n+1\}})
+\lambda_1^{P}\overline{\lambda_1^{P}}\det(\psi({\bf M})^{\{1,n+1\}\{1,n+1\}})\\
&=|\lambda_1|^2\det(\psi({\bf M})^{\{1,n+1\}\{1,n+1\}})\\
&=|\lambda_1|^2|\lambda_2|^2\det(\psi({\bf M})^{\{1,2,n+1,n+2\}\{1,2,n+1,n+2\}})={\cdots}\\
&=\prod_{r=1}^n|\lambda_r|^2.
\end{split}
\end{equation*}

\noindent
(vii)\ \ Let ${\bf A}$ be an $m \times n$ matrix and ${\bf B}$ an $n \times m$ matrix.
Then by the definition of the Study determinant, we have 
\begin{equation*}
\begin{split}
\Sdet({\bf I}_m-{\bf A}{\bf B})&={\rm Sdet}(\psi({\bf I}_m -{\bf A}{\bf B})) 
=\det({\bf I}_{2m}-\psi({\bf A})\psi({\bf B})) \\
&=\det({\bf I}_{2n}-\psi({\bf B})\psi({\bf A})) 
=\det(\psi({\bf I}_n-{\bf B}{\bf A}))\\
&=\Sdet({\bf I}_n-{\bf B}{\bf A}). 
\end{split}
\end{equation*}
Here the equation $\det({\bf I}_{2m}-\psi({\bf A})\psi({\bf B}))
=\det({\bf I}_{2n}-\psi({\bf B})\psi({\bf A}))$ is based on the 
property of determinant that can be found in, for example, 
Problem 5 on page 47 in \cite{Zhang2011}. 
\end{proof}

\begin{remark}
$\Sdet$ is not multilinear as $\det$ is. Furthermore, 
$\Sdet({}^T\!{\bf M})=\Sdet({\bf M})$ does not hold in general where 
${}^T\!{\bf M}$ is the transpose of ${\bf M}$. 
\end{remark}

\section{The quaternionic second weighted zeta function of a graph} 
We follow symbols and notations in the previous section. 
We shall give the definition of a quaternionic analogue of the second weighted zeta function 
and derive its determinant expression of Bass type by the Study determinant. 
In the same way as the complex case, 
consider a quaternionic matrix 
${\bf W} =({\bf W}_{uv})_{u,v{\in}V(G)}{\;\in\;}\Mat(n,\HM)$ with $(u,v)$-entry 
equals $0$ if $(u,v){\;\notin\;}D(G)$. 
We call ${\bf W}$ a {\em quaternionic weighted matrix} of $G$.
Furthermore, let $w(u,v)= {\bf W}_{uv}$ for $u,v \in V(G)$ and 
$w(e)= w(u,v)$ if $e=(u,v) \in D(G)$. 
For each path $P=( e_1 , \cdots , e_{\ell} )$ of $G$, the {\em norm} 
$w(P)$ of $P$ is defined by 
$w(P)=w(e_1)w(e_2){\cdots}w(e_{\ell})$.

For a quaternionic weighted matrix ${\bf W}$ of $G$, the $2m \times 2m$ 
quaternionic matrix 
${\bf B}_w=( {\bf B}^{(w)}_{ef} )_{e,f \in D(G)}{\;\in\;}\Mat(2m,\HM)$ 
is defined as follows: 
\begin{equation}\label{DefBw}
{\bf B}^{(w)}_{ef} =\left\{
\begin{array}{ll}
w(f) & \mbox{if $t(e)=o(f)$, } \\
0 & \mbox{otherwise.}
\end{array}
\right.
\end{equation}
We define the {\em quaternionic second weighted zeta function} of $G$ to be 
as follows: 
\[
{\bf Z}^{\HM}_1 (G,w,t)= \Sdet ( {\bf I}_{2m} -t ( {\bf B}_w - {\bf J}_0 ) )^{-1},  
\]
where $t$ is a quaternionic variable.

The Study determinant expression of Bass type for the quaternionic second 
weighted zeta function of a graph is given as follows. 
The block diagonal sum \( {\bf M}_{1} \oplus \cdots \oplus {\bf M}_{s} \) 
of square matrices 
${\bf M}_{1}, \cdots , {\bf M}_{s}$ is defined as the square matrix: 
\[
\left[
\begin{array}{ccc}
{\bf M}_{1} & \: & 0 \\
\: & \ddots & \: \\
0 & \: & {\bf M}_{s} 
\end{array}
\right].
\]
If \( {\bf M}_{1} = {\bf M}_{2} = \cdots = {\bf M}_{s} = {\bf N} \),
then we write 
\( s \circ {\bf N} = {\bf M}_{1} \oplus \cdots \oplus {\bf M}_{s} \).
The {\it Kronecker product} $ {\bf A} \bigotimes {\bf B} $
of matrices {\bf A} and {\bf B} is considered as the matrix 
{\bf A} having the element $a_{rs}$ replaced by the matrix $a_{rs} {\bf B}$.

\begin{theorem}\label{DetExpressionOfBassType}
Let $G$ be a connected graph, and 
let ${\bf W}$ be a quaternionic weighted matrix of $G$. 
Then the reciprocal of the quaternionic second weighted zeta function of $G$ is given by 
\[
{\bf Z}^{\HM}_1 (G,w,t )^{-1} =|1- t^2 |^{2m-2n} 
{\rm Sdet} ({\bf I}_n - {\bf W}t+( {\bf D}_w - {\bf I}_n ) t^2 ) , 
\]
where $n=|V(G)|$, $m=|E(G)|$ and 
${\bf D}_w=({\bf D}^{(w)}_{uv})_{u,v{\in}V(G)}$ is the diagonal matrix 
with ${\bf D}^{(w)}_{uu} = \sum_{e:o(e)=u}w(e)$. 
\end{theorem}

\begin{proof} Let 
$D(G)= \{ f_1 , \cdots , f_{m} , f^{-1}_1 , \cdots , f^{-1}_{m} \} $. 
Arrange arcs of $G$ as follows: 
\[
f_{1}, f^{-1}_1 , \cdots , f_{m}, f^{-1}_{m} . 
\] 
By the definition of the second weighted zeta function of $G$ and 
Proposition \ref{SdetProperties}, we have 
\begin{equation}\label{ReciprocalSecondZeta}
\begin{split}
{\bf Z}^{\HM}_1 (G,w,t )^{-1} & = {\Sdet}({\bf I}_{2m}-t({\bf B}_w-{\bf J}_0)) \\ 
& = {\Sdet}({\bf I}_{2m}+t {\bf J}_0-t{\bf B}_w) \\
& = {\Sdet} ({\bf I}_{2m}-t{\bf B}_w({\bf I}_{2m}+t {\bf J}_0)^{-1}) 
{\Sdet}({\bf I}_{2m}+t{\bf J}_0 ) . 
\end{split}
\end{equation}

Let $t= t_s +j t_p \in {\HM}$ be the symplectic decomposition. 
Then we have 
\begin{equation}\label{Blocks}
{\rm Sdet} ( {\bf I}_{2m} +t {\bf J}_0 )= 
\det  
\left[ 
\begin{array}{cc}
{\bf I}_{2m} + t_s {\bf J}_0 & - \overline{t_p} {\bf J}_0 \\ 
t_p {\bf J}_0 & {\bf I}_{2m} + \overline{t_s} {\bf J}_0   
\end{array} 
\right],
\end{equation}
where ${\bf I}_{2m} + t_s {\bf J}_0 $ and $t_p {\bf J}_0 $ are given by
\[
{\bf I}_{2m} + t_s {\bf J}_0 = m \circ \left[ 
\begin{array}{cc}
1 & t_s \\
t_s & 1   
\end{array} 
\right] 
\ , \  
t_p {\bf J}_0 = m \circ \left[ 
\begin{array}{cc}
0 & t_p \\
t_p & 0   
\end{array} 
\right] 
. 
\]
For any two complex numbers $\alpha$ and $\beta$, 
\[
\left[ 
\begin{array}{cc}
1 & \alpha \\
\alpha & 1   
\end{array} 
\right] 
\  
\left[ 
\begin{array}{cc}
0 & \beta \\
\beta & 0   
\end{array} 
\right] 
= 
\left[ 
\begin{array}{cc}
0 & \beta \\
\beta & 0   
\end{array} 
\right] 
\ 
\left[ 
\begin{array}{cc}
1 & \alpha \\
\alpha & 1   
\end{array} 
\right] 
\]
holds. This implies any two blocks in the right-hand side of (\ref{Blocks}) commute. 
Thus by Lemma \ref{DetPartitionedMat}, we have 
\begin{equation}\label{SdetIPlusJ}
\begin{split}
{\Sdet}({\bf I}_{2m}+t{\bf J}_0) & = \det(({\bf I}_{2m}+t_s {\bf J}_0)({\bf I}_{2m}+\overline{t_s} {\bf J}_0 )+ t_p \overline{t_p} {\bf J}^2_0 ) \\
& = \det ( {\bf I}_{2m} +( t_s + \overline{t_s} ) {\bf J}_0 +(|t_s|^2 +|t_p|^2 ) {\bf I}_{2m} ) \\
& = \det 
{\bf I}_m \bigotimes 
\left[ 
\begin{array}{cc}
1+|t|^2 & 2\RP\,t \\ 
2\RP\,t & 1+|t|^2  
\end{array} 
\right] 
\\ 
& = \{ (1+|t|^2 )^2 -4(\RP\,t)^2 \} {}^m \\
& = \{ (1+tt^*)^2 -(t+t^*)^2 \}^m \\
& = \{ (1-t^2)(1-(t^*)^2) \}^m \\
& = \{ (1-t^2)(1-t^2)^* \}^m \\
& = |1-t^2|^{2m}
\end{split}
\end{equation}
On the other hand, since the following holds: 
\begin{equation}\label{CommtativityOft}
t\dfrac{1}{1-t^2}=t\dfrac{(1-(t^2))^*}{|1-t^2|^2}=t\dfrac{1-t^*t^*}{|1-t^2|^2} 
=\dfrac{1-t^*t^*}{|1-t^2|^2}t=\dfrac{1}{1-t^2}t,
\end{equation}
and hence
\begin{equation}\label{InverseOfIPlusJ}
({\bf I}_{2m}+t{\bf J}_0)({\bf I}_{2m}-t{\bf J}_0)\dfrac{1}{1-t^2}
=({\bf I}_{2m}-t{\bf J}_0)\dfrac{1}{1-t^2}({\bf I}_{2m}+t{\bf J}_0)
={\bf I}_{2m}. 
\end{equation}
Thus, we have 
\[
({\bf I}_{2m}+t{\bf J}_0)^{-1} =({\bf I}_{2m}-t{\bf J}_0)\dfrac{1}{1-t^2}. 
\]
Now, let ${\bf K} =( {\bf K}_{ev} )$ ${}_{e \in D(G), v \in V(G)} $ 
and ${\bf L} =( {\bf L}_{ev} )_{e \in D(G), v \in V(G)} $ be 
$2m \times n$ matrices defined as follows: 
\[
{\bf K}_{ev} =\left\{
\begin{array}{ll}
w(e) & \mbox{if $o(e)=v$, } \\
0 & \mbox{otherwise, } 
\end{array}
\right.
{\bf L}_{ev} =\left\{
\begin{array}{ll}
1 & \mbox{if $t(e)=v$, } \\
0 & \mbox{otherwise. } 
\end{array}
\right.
\] 
Then 
\[
{\bf L} {}^T\!{\bf K} = {\bf B}_w
\]
holds, where ${}^T\!{\bf K}$ is the transpose of ${\bf K}$. 
Thus, by Lemma \ref{SdetProperties}, (\ref{CommtativityOft}) 
and (\ref{InverseOfIPlusJ}), 
we can show that  
\begin{equation*}
\begin{split}
&{\Sdet}({\bf I}_{2m}-t{\bf B}_w({\bf I}_{2m}+t{\bf J}_0)^{-1})   
={\Sdet}({\bf I}_{2m}-t{\bf L}{}^T\!{\bf K}({\bf I}_{2m}+t{\bf J}_0)^{-1}) \\  
& = {\Sdet}({\bf I}_{n}-{}^T\!{\bf K}({\bf I}_{2m}+t{\bf J}_0)^{-1}(t{\bf L})) 
={\Sdet}({\bf I}_{n}-{}^T\!{\bf K}({\bf I}_{2m}+t{\bf J}_0)^{-1}{\bf L}t) \\ 
& = {\Sdet}({\bf I}_{n}-{}^T\!{\bf K}({\bf I}_{2m}-t{\bf J}_0)\dfrac{1}{1-t^2}
{\bf L}t) 
= {\Sdet}({\bf I}_{n}-{}^T\!{\bf K}({\bf I}_{2m}-t {\bf J}_0){\bf L}t\dfrac{1}{1-t^2}). 
\end{split}
\end{equation*}
For an arc $(u,v) \in D(G)$, 
\[
({}^T\!{\bf K}({\bf I}_{2m} -t {\bf J}_0 ) {\bf L} )_{uv} 
=w(u,v) .  
\]
In the case of $u=v$,  
\[
({}^T\!{\bf K}({\bf I}_{2m} -t {\bf J}_0 ) {\bf L} )_{uu} 
=- \sum_{o(e)=u} w(e)t . 
\]
Thus by Proposition \ref{SdetProperties}, we have 
\begin{equation*}
\begin{split}
{\Sdet}({\bf I}_{2m}-t{\bf B}_w({\bf I}_{2m}+t {\bf J}_0)^{-1}) 
& = {\Sdet}({\bf I}_{n}-({\bf W}-{\bf D}_w t)t\dfrac{1}{1-t^2}) \\
& = {\Sdet}(((1-t^2 ) {\bf I}_{n}-{\bf W}t+{\bf D}_w t^2 )\dfrac{1}{1-t^2}) \\
& = {\Sdet}({\bf I}_{n}-{\bf W}t+({\bf D}_w-{\bf I}_n)t^2 )
\Big{|}\dfrac{1}{1-t^2}\Big{|}^{2n}.  
\end{split}
\end{equation*}
Finally, we conclude from (\ref{ReciprocalSecondZeta}) and (\ref{SdetIPlusJ}) that
\begin{equation*}
\begin{split}
{\bf Z}^{\HM}_1 (G,w,t)^{-1} 
& = {\Sdet}({\bf I}_{2m}-t{\bf B}_w({\bf I}_{2m}+t{\bf J}_0 )^{-1})
{\rm Sdet}({\bf I}_{2m}+t{\bf J}_0) \\ 
& = {\Sdet}({\bf I}_{n}-{\bf W}t+({\bf D}_w-{\bf I}_n)t^2)
\Big{|}\dfrac{1}{1-t^2}\Big{|}^{2n} 
|1-t^2|^{2m} \\ 
& = |1- t^2 |^{2m-2n} 
{\Sdet}({\bf I}_{n}-{\bf W}t+({\bf D}_w-{\bf I}_n)t^2). 
\end{split}
\end{equation*}
\end{proof}

\section{The Euler product for the quaternionic second \\ weighted zeta function}
In this section, we derive the Euler product of the quaternionic second weighted zeta 
function. 
In order to obtain the Euler product, we make use of the notion of noncommutative 
formal power series. 
Our argument on this subject is based on the proofs of Amitsur's identity in 
\cite{RS1987} or \cite{FZ1999} until (\ref{AmitsurForMatrix}). 
For the sake of argument, we will give a brief account of noncommutative 
formal power series at first. 
A detailed exposition of formal series can be found in \cite{BR2011}. 
Let $X=\{x_1,{\cdots},x_N\}$ be a finite nonempty totally ordered set in which 
elements are arranged ascendingly. 
$X^*$ denotes the free monoid generated by 
$X$. Let $<$ be the lexicographic order on $X^*$ derived from the total order on $X$. 
For a word $w=x_{i_1}x_{i_2}{\cdots}x_{i_d}{\;\in\;}X^*$, 
$d$ is called the {\it length} of $w$ which is denoted by $|w|$. 
The length of the empty word is defined to be $0$. 
A nonempty word $w$ in $X^*$ is called a {\it Lyndon word} if $w$ is {\it prime}, 
namely, not a power $w'^r$ of any other word $w'$ for any $r\geq2$, and 
is minimal in the cyclic rearrangements of $w$. We denote by $L_X$ the set of Lyndon words 
in $X^*$. 
It is well known that any nonempty word $w$ can be formed uniquely 
as a nonincreasing sequence of Lyndon words. 

\begin{theorem}\label{LyndonFactorization}
For any nonempty word $w{\;\in\;}X^*$, there exists a unique nonincreasing sequence of 
Lyndon words $l_1,l_2,\cdots,l_d$ such that $w=l_1l_2{\cdots}l_d$. 
\end{theorem}

\begin{proof} For the proof, see for example \cite{Lothaire1997}. 
\end{proof}

Let us consider the ring of noncommutative formal power series $\RM{\laangle}X^*{\raangle}$. 
Each element $f$ of $\RM{\laangle}X^*{\raangle}$ is displayed as 
\[
f=\sum_{w{\in}X^*}f_w w \quad (f_w{\;\in\;}\RM).
\]
$\RM{\laangle}X^*{\raangle}$ can be equipped with the topology defined by 
the following manner. 
Let $\omega$ be the function defined as follows: 
\begin{equation*}\label{UltrametricDistance}
\begin{split}
\omega :\,&\RM{\laangle}X^*{\raangle}{\times}\RM{\laangle}X^*{\raangle}
{\longrightarrow}\NM{\;\cup\;}\{\infty\}=\{0,1,2,{\cdots},\infty\} \\
& \omega(\alpha,\beta)=\begin{cases}
\infty \quad  
\text{if $\{w{\;\in\;}X^*\;|\;\alpha_w{\;\neq\;}\beta_w \}=\phi $},\\
\inf \{n{\;\in\;}\NM\;|\;\exists w{\;\in\;}X^*,\,|w|=n,\, 
\alpha_w{\;\neq\;}\beta_w \} \quad \text{otherwise}.
\end{cases}
\end{split}
\end{equation*}
Then an ultrametric distance $d_{\omega}$ on $\RM{\laangle}X^*{\raangle}$ 
is given by $d_{\omega}(\alpha,\beta)=2^{-\omega(\alpha,\beta)}$ and a topology on 
$\RM{\laangle}X^*{\raangle}$ is derived from $d_{\omega}$. 
We notice that $\RM{\laangle}X^*{\raangle}$ is complete for this topology. 
Since $(1-l)^{-1}=1+l+l^2+{\cdots}$ for every $l{\;\in\;}X^*$ 
in $\RM{\laangle}X^*{\raangle}$, Theorem \ref{LyndonFactorization} implies 
\begin{equation}\label{SumOfWords1}
\prod_{l{\in}L_X}^{>}(1-l)^{-1}=\sum_{w{\in}X^*}w,
\end{equation}
in $\RM{\laangle}X^*{\raangle}$, 
where $\displaystyle \prod_{l{\in}L_X}^{>}$ means that the factors are multiplied 
in decreasing order. 
On the other hand, it follows that 
\begin{equation}\label{SumOfWords2}
\sum_{w{\in}X^*}w=\{1-(x_1+{\cdots}+x_N)\}^{-1}.
\end{equation}
(\ref{SumOfWords1}) and (\ref{SumOfWords2}) imply the following equation: 
\begin{equation}\label{InveseOfAmitsurIdentity}
\{1-(x_1+{\cdots}+x_N)\}^{-1}=\prod_{l{\in}L_X}^{>}(1-l)^{-1}.
\end{equation}
From (\ref{InveseOfAmitsurIdentity}), we obtain 
\begin{proposition}
\begin{equation}\label{AmitsurIdentity}
1-(x_1+{\cdots}+x_N)=\prod_{l{\in}L_X}^{<}(1-l),
\end{equation}
where $\displaystyle \prod_{l{\in}L_X}^{<}$ means that the factors are multiplied 
in increasing order. 
\end{proposition}
\begin{proof} In order to show that 
\begin{equation}\label{ProdEquals1}
\Big{\{}\prod_{l{\in}L_X}^{>}(1-l)^{-1}\Big{\}}
\Big{\{}\prod_{l{\in}L_X}^{<}(1-l)\Big{\}}=1,
\end{equation}
we check that for an arbitrary nonnegative integer $r{\;\geq\;}0$  
the sum of words of length $r$ equals 
$1$ if $r=0$ and $0$ if $r>0$. Since 
$\prod_{l{\in}L_X}^{>}(1-l)^{-1}=
\prod_{l{\in}L_X}^{>}(1+l+l^2+{\cdots})$, 
the sum of words of length at most $d$ in the left hand side of 
(\ref{ProdEquals1}) is the same as that of the product: 
\[
\Big{\{}\prod_{\substack{l{\in}L_X\\|l|{\leq}d}}^{>}
(1+l+l^2+{\cdots})\Big{\}}
\Big{\{}\prod_{\substack{l{\in}L_X\\|l|{\leq}d}}^{<}(1-l)\Big{\}}.
\]
This is a finite product since $|X|<\infty$ and therefore 
is equal to $1$. Since $d$ is arbitrary, (\ref{ProdEquals1}) holds. 
\end{proof}

Let $[2m]=\{1,2,\cdots,2m\}$ and $[2m]{\times}[2m]$ the 
cartesian product with the lexicographic order derived from 
the natural order on $[2m]$. 
We say that a word $w=(i_1,j_1)(i_2,j_2){\cdots}(i_d,j_d){\;\in\;}$ $([2m]{\times}[2m])^*$ is 
{\it connected} if $j_r=i_{r+1}$ for $r=1,2,{\cdots},d-1$. For a connected word 
$w=(i_1,i_2)(i_2,i_3){\cdots}(i_d,j_d)\;\in\;([2m]{\times}[2m])^*$, we set 
$o(w)=i_1$ and $t(w)=j_d$. 
Consider the finite nonempty set $X=\{x(r,s)\;|\;(r,s){\;\in\;}[2m]{\times}[2m]\}$ 
equipped with the total order derived from $[2m]{\times}[2m]$. 
For each matrix ${\bf A}=(a_{rs}){\;\in\;}\Mat(2m,\HM)$, we define $\rho^{\bf A}$ to be the 
$\RM$-algebra homomorphism from the monoid ring $\RM[X^*]$ 
to $\Mat(2m,\HM)$ defined by $\rho^{\bf A}(x(r,s))=a_{rs}{\bf E}_{rs}$, where ${\bf E}_{rs}$ denotes the 
$(r,s)$-matrix unit. Let ${\bf A}(r,s)=a_{rs}{\bf E}_{rs}$. 
For $a_{rs}\;(1{\;\leq\;}r,s{\;\leq\;}2m)$ with $|a_{rs}|$ sufficiently small, 
we can apply $\rho^{\bf A}$ to (\ref{AmitsurIdentity}) so that 
\begin{equation}\label{AmitsurForMatrix}
{\bf I}_{2m}-\{{\bf A}(1,1)+{\bf A}(1,2)+{\cdots}+{\bf A}(2m,2m)\}
=\prod_{l{\in}L_{[2m]{\times}[2m]}}^{<}({\bf I}_{2m}-{\bf A}_l), 
\end{equation}
where ${\bf A}_l={\bf A}(i_1,j_1){\bf A}(i_2,j_2){\cdots}{\bf A}(i_d,j_d)$ for each 
$l=(i_1,j_1)(i_2,j_2){\cdots}(i_d,j_d)$ ${\;\in\;}L_{[2m]{\times}[2m]}$. 
Indeed, the following holds: 

\begin{proposition}\label{ConvergenceOfInfiniteProduct}
For $a_{rs}\;(1{\;\leq\;}r,s{\;\leq\;}2m)$ with $|a_{rs}|$ sufficiently small, 
all entries in the right hand side of {\rm (\ref{AmitsurForMatrix})} converge absolutely with 
respect to the norm of quaternions. 
Particularly all entries in the right hand side of {\rm (\ref{AmitsurForMatrix})} converge. 
\end{proposition}

\begin{proof}
We put $a_l= a_{i_1j_1}a_{i_2j_2}{\cdots}a_{i_dj_d}$ and 
${\bf E}_l= {\bf E}_{i_1j_1}{\bf E}_{i_2j_2}{\cdots}{\bf E}_{i_dj_d}$ for a Lyndon word 
$l=(i_1,j_1)(i_2,j_2){\cdots}(i_d,j_d){\;\in\;}L_{[2m]{\times}[2m]}$. 
Since ${\bf A}_l={\bf A}_{(i_1,j_1){\cdots}(i_d,j_d)}=
a_{i_1i_2}a_{i_2i_3}{\cdots}a_{i_{d-1}i_d}a_{i_dj_d}{\bf E}_l$, 
it follows that 
\begin{equation}\label{EqnInfinteSeriesOfMatirx}
\begin{split}
\prod_{l{\in}L_{[2m]{\times}[2m]}}^{<}({\bf I}_{2m}-{\bf A}_l)
&=\prod_{l{\in}L_{[2m]{\times}[2m]}}^{<}({\bf I}_{2m}-a_l{\bf E}_l)\\
&={\bf I}_{2m}+\sum_{h=1}^{\infty}
\sum_{\substack{w=l_1{\cdots}l_{n(w)},\,l_1<{\cdots}<l_{n(w)}\\
l_1,{\cdots},l_{n(w)}{\in}L_{[2m]{\times}[2m]},\,|w|=h}}
(-1)^{n(w)}a_w{\bf E}_w,
\end{split}
\end{equation}
where 
$a_w=a_{l_1}a_{l_2}{\cdots}a_{l_{n(w)}},\,
{\bf E}_w={\bf E}_{l_1}{\bf E}_{l_2}{\cdots}{\bf E}_{l_{n(w)}}$ and 
$n(w)$ is the number of Lyndon words that are multiplied in $w$. 
If $|w|=h$ then $a_w$ can be expressed by $a_w=a_{i_1j_1}{\cdots}a_{i_hj_h}$. 
Then we notice that $|a_w|=|a_{i_1j_1}|{\cdots}|a_{i_hj_h}|$. 
We can easily see that ${\bf E}_w={\bf E}_{o(w)t(w)}$ if $w$ is connected, and 
${\bf E}_w={\bf O}_{2m}$ otherwise. 
Therefore $(r,s)$-entry of (\ref{EqnInfinteSeriesOfMatirx}) is expressed as follows: 
\begin{equation}\label{rsEntryOfInfiniteProduct}
\Big{(}\prod_{l{\in}L_{[2m]{\times}[2m]}}^{<}({\bf I}_{2m}-{\bf A}_l)\Big{)}_{rs}
=\delta_{rs}+\sum_{h=1}^{\infty}\sum_{\substack{w=l_1{\cdots}l_{n(w)},\,l_1<{\cdots}<l_{n(w)}\\
l_1,{\cdots},l_{n(w)}{\in}L_{[2m]{\times}[2m]},\,|w|=h\\
w\text{ is connected},\,o(w)=r,t(w)=s}}(-1)^{n(w)}a_w.
\end{equation}
In (\ref{rsEntryOfInfiniteProduct}), 
the number of $w$ of length $h$ is at most the number of words in $([2m]{\times}[2m])^*$ 
of length $h$ which equals $(2m)^{2h}$. 
Hence if $|a_{rs}|<1/(8m^2)$ for all $r,s=1,{\cdots},2m$, then 
\begin{equation*}\begin{split}
&|\delta_{rs}|+\sum_{h=1}^{\infty}\Big{|}\sum_{\substack{w=l_1{\cdots}l_{n(w)},\,l_1<{\cdots}<l_{n(w)}\\
l_1,{\cdots},l_{n(w)}{\in}L_{[2m]{\times}[2m]},\,|w|=h\\
w\text{ is connected},\,o(w)=r,t(w)=s}}(-1)^{n(w)}a_w\Big{|}\\
&{\;\leq}|\delta_{rs}|+\sum_{h=1}^{\infty}\sum_{\substack{w=l_1{\cdots}l_{n(w)},\,l_1<{\cdots}<l_{n(w)}\\
l_1,{\cdots},l_{n(w)}{\in}L_{[2m]{\times}[2m]},\,|w|=h\\
w\text{ is connected},\,o(w)=r,t(w)=s}}|a_w|\\
&<1+\sum_{h=1}^{\infty}\dfrac{(2m)^{2h}}{(8m^2)^h}=1+\sum_{h=1}^{\infty}\dfrac{1}{2^h}
=2
\end{split}\end{equation*} 
Thus the right hand side of (\ref{rsEntryOfInfiniteProduct}) 
converges absolutely with respect to the norm of quaternions. 
\end{proof}

Since ${\bf A}(1,1)+{\bf A}(1,2)+{\cdots}+{\bf A}(2m,2m)={\bf A}$, 
it follows from (\ref{AmitsurForMatrix}) that 

\begin{proposition}\label{AmitsurForMatrixThm}
Let ${\bf A}=(a_{rs})$ be a $2m{\times}2m$ quaternionic matrix with $|a_{rs}|$ 
sufficiently small. Then 
\begin{equation}\label{EqnAmitsurIdentity}
{\bf I}_{2m}-{\bf A}=\prod_{\substack{(i_1,j_1){\cdots}(i_d,j_d){\in}L_{[2m]{\times}[2m]}\\
j_r=i_{r+1}\;(r=1,{\cdots},d-1)}}^{<}({\bf I}_{2m}-a_{i_1i_2}a_{i_2i_3}{\cdots}
a_{i_{d-1}i_d}a_{i_dj_d}{\bf E}_{i_1j_d}),
\end{equation}
in $\Mat(2m,\HM)$. 
\end{proposition}

Now we take Study determinants of both sides in 
(\ref{EqnAmitsurIdentity}). 
\begin{equation}\label{EqnSdetIdentity}
\begin{split}
&\Sdet({\bf I}_{2m}-{\bf A})\\
&=\Sdet\Big{(}\prod_{\substack{(i_1,j_1){\cdots}(i_d,j_d){\in}L_{[2m]{\times}[2m]}\\
j_r=i_{r+1}\;(r=1,{\cdots},d-1)}}^{<}({\bf I}_{2m}-a_{i_1i_2}a_{i_2i_3}{\cdots}
a_{i_{d-1}i_d}a_{i_dj_d}{\bf E}_{i_1j_d})\Big{)}\\
&=\prod_{\substack{(i_1,j_1){\cdots}(i_d,j_d){\in}L_{[2m]{\times}[2m]}\\
j_r=i_{r+1}\;(r=1,{\cdots},d-1)}}\Sdet({\bf I}_{2m}-a_{i_1i_2}a_{i_2i_3}{\cdots}
a_{i_{d-1}i_d}a_{i_dj_d}{\bf E}_{i_1j_d}).
\end{split}
\end{equation}
We notice that the last formula does not depend on the order in which 
factors are multiplied since Study determinants take values in $\RM$. 
It follows from Proposition \ref{SdetProperties} (vi) that if $j_d=i_1$, then 
\begin{equation*}
\Sdet({\bf I}_{2m}-a_{i_1i_2}a_{i_2i_3}{\cdots}a_{i_di_1}
{\bf E}_{i_1i_1})=|1-a_{i_1i_2}a_{i_2i_3}{\cdots}a_{i_di_1}|^2,
\end{equation*}
and otherwise,  
\[
\Sdet({\bf I}_{2m}-a_{i_1i_2}a_{i_2i_3}{\cdots}a_{i_dj_d}
{\bf E}_{i_1j_d})=1.
\]
Let ${\bf W}=({\bf W}_{uv})_{u,v{\in}V(G)}$ be an arbitrary quaternionic 
weighted matrix of $G$ and $t$ a quaternion with $|t|$ sufficiently small so that 
$|t\tilde{w}(e,f)|<1/(8m^2)$ for all $e,f{\;\in\;}D(G)$, 
where $\tilde{w}(e,f)$ is as in (\ref{EntriesOf2ndEdgeMatrix}). 
Putting ${\bf A}=t({\bf B}_w-{\bf J}_0)$ and indexing rows and columns with 
$e_1,e_2,{\cdots},e_{2m}{\;\in\;}D(G)$, then we have 
$a_{rs}=a_{e_re_s}=t\tilde{w}(e_r,e_s)$. 
Therefore, (\ref{EqnSdetIdentity}) yields 
\begin{equation*}
\begin{split}
&\Sdet({\bf I}_{2m}-t({\bf B}_w-{\bf J}_0))\\
&=\prod_{(i_1,i_2){\cdots}(i_d,i_1){\in}L_{[2m]{\times}[2m]}}
\Sdet({\bf I}_{2m}-t\tilde{w}(e_{i_1},e_{i_2})t\tilde{w}(e_{i_2},e_{i_2}){\cdots}
t\tilde{w}(e_{i_d},e_{i_1}){\bf E}_{i_1i_1})\\
&=\prod_{(i_1,i_2){\cdots}(i_d,i_1){\in}L_{[2m]{\times}[2m]}}
|1-t\tilde{w}(e_{i_1},e_{i_2})t\tilde{w}(e_{i_2},e_{i_2}){\cdots}
t\tilde{w}(e_{i_d},e_{i_1})|^2
\end{split}
\end{equation*}
Each Lyndon word $(i_1,i_2){\cdots}(i_d,i_1)$ in $L_{[2m]{\times}[2m]}$ 
corresponds to a Lyndon word 
$i_1i_2{\cdots}i_d$ in $L_{[2m]}$ bijectively. Hence we obtain 

\begin{theorem}\label{ThmEulerProduct}
Let $t$ be a quaternion with $|t|$ sufficiently small. Then
\begin{equation*}
\begin{split}
{\bf Z}^{\HM}_1(G,w,t)
=\prod_{i_1i_2{\cdots}i_d{\in}L_{[2m]}}
|1-t\tilde{w}(e_{i_1},e_{i_2})t\tilde{w}(e_{i_2},e_{i_2}){\cdots}
t\tilde{w}(e_{i_d},e_{i_1})|^{-2}.
\end{split}
\end{equation*}
\end{theorem}

Since real numbers are central in $\HM$, it follows that  

\begin{corollary}\label{CorEulerProduct}
Let $t$ be a real number with $|t|$ sufficiently small. Then
\begin{equation*}
\begin{split}
{\bf Z}^{\HM}_1(G,w,t)
=\prod_{i_1i_2{\cdots}i_d{\in}L_{[2m]}}
|1-\tilde{w}(e_{i_1},e_{i_2})\tilde{w}(e_{i_2},e_{i_2}){\cdots}
\tilde{w}(e_{i_d},e_{i_1})t^d|^{-2}.
\end{split}
\end{equation*}
\end{corollary}




\section*{Acknowledgments}
The first author is partially supported by the Grant-in-Aid for Scientific Research 
(Challenging Exploratory Research) of Japan Society for the Promotion of Science (Grant No. 15K13443).
The third author is partially supported by the Grant-in-Aid for Scientific Research 
(C) of Japan Society for the Promotion of Science (Grant No. 15K04985).

\begin{thebibliography}{99}

\bibitem{Aslaksen1996}
Aslaksen, H.: 
Quaternionic Determinants. 
Math. intelligencer {\bf 18}, no3, pp. 57--65 (1996)

\bibitem{Bass1992}
Bass, H.: 
The Ihara-Selberg zeta function of a tree lattice. 
Internat. J. Math. {\bf 3}, pp. 717--797 (1992)

\bibitem{BR2011}
Berstel, J., Reutenauer, C.:  
Noncommutative rational series with applications, 
Cambridge University press, Cambridge, 2011. 

\bibitem{FZ1999}
Foata, D., Zeilberger, D.:  
A combinatorial proof of Bass's evaluations of the Ihara-Selberg zeta 
function for graphs.  
Trans. Amer. Math. Soc. {\bf 351}, pp. 2257--2274 (1999) 

\bibitem{Hashimoto1989}
Hashimoto, K.:
Zeta Functions of Finite Graphs and Representations 
of $p$-Adic Groups. 
in "Adv. Stud. Pure Math". Vol. 15, pp. 211--280, Academic Press, 
New York, 1989.

\bibitem{Hashimoto1990}
Hashimoto, K.:
On Zeta and L-Functions of Finite Graphs. 
Internat. J. Math. {\bf 1}, pp. 381--396 (1990)

\bibitem{HKSS2013}
Higuchi, H., Konno, N., Sato, I., Segawa, E.: 
A note on the discrete-time evolutions of quantum walk on a graph. 
J. Math-for-Industry, {\bf 5} (2013B-3), pp.103--109, (2013)

\bibitem{Ihara1966}
Ihara, Y.: 
On discrete subgroups of the two by two projective linear group 
over $p$-adic fields. 
J. Math. Soc. Japan {\bf 18}, pp.  219--235 (1966)

\bibitem{Konno2015} 
Konno, N.: Quaternionic quantum walks. 
Quantum Stud.: Math. Found. {\bf 2}, 63--76 (2015)

\bibitem{KMS2016}
Konno, N., Mitsuhashi, H., Sato, I.: 
The discrete-time quaternionic quantum walk on a graph. 
Quantum Inf. Process. {\bf 15}, 651--673 (2016) 

\bibitem{KS2012}
Konno, N., Sato, I.: 
On the relation between quantum walks and zeta functions. 
Quantum Information Processing {\bf 11} Issue 2, pp. 341--349 (2012) 

\bibitem{KS2000}
Kotani, M., Sunada, T.:
Zeta functions of finite graphs. 
J. Math. Sci. U. Tokyo {\bf 7}, pp.  7--25 (2000) 

\bibitem{Lothaire1997}
Lothaire, M.: Combinatorics on words. 
Cambridge University Press, (1997)

\bibitem{MS2004}
Mizuno, H., Sato, I.:
Weighted zeta functions of graphs. 
J. Combin. Theory Ser. B {\bf 91}, pp.169--183 (2004) 

\bibitem{MS2008}
Mizuno, H., Sato, I. :
The scattering matrix of a graph. Electron. J. Combin.  {\bf 15}, R96 (2008) 

\bibitem{RS1987}
Reutenauer, C., Sch\"{u}tzenberger, M-P.:  
A formula for the Determinant of a Sum of Matrices.  
Lett. Math. Phys. {\bf 13}, pp. 299--302 (1987) 

\bibitem{Sato}
Sato, I. :
A new Bartholdi zeta function of a graph. Int. J. Algebra {\bf 1}, pp. 269--281 (2007) 

\bibitem{Serre}
Serre, J. -P.:  
Trees, 
Springer-Verlag, New York, 1980. 

\bibitem{ST1996}
Stark, H. M., Terras, A. A.: 
Zeta functions of finite graphs and coverings. 
Adv. Math. {\bf 121}, pp. 124-165 (1996) 

\bibitem{Study1920}
Study, E.: 
Zur Theorie der lineare Gleichungen. 
Acta. Math. {\bf 42}, pp. 1-61 (1920) 

\bibitem{Smilansky2007}
Smilansky, U.: Quantum chaos on discrete graphs. 
J. Phys. A: Math. Theor. {\bf 40}, F621--F630 (2007)

\bibitem{Sunada1986}
Sunada, T.: 
$L$-Functions in Geometry and Some Applications.  
in "Lecture Notes in Math"., Vol. 1201, pp. 266--284, 
Springer-Verlag, New York, 1986.

\bibitem{Sunada1988}
Sunada, T.: 
Fundamental Groups and Laplacians(in Japanese),
Kinokuniya, Tokyo, 1988.

\bibitem{Zhang1997}
Zhang, F.: Quaternions and matrices of quaternions. 
Linear Algebra Appl. {\bf 251}, pp. 21--57 (1997)

\bibitem{Zhang2011}
Zhang, F.: 
Matrix Theory 2nd ed.,
Springer, New York, 2011.

\end{thebibliography}

\end{document}